\documentclass[a4paper,12pt]{article}
\usepackage{amsthm}
\usepackage{amsmath,amssymb,latexsym,amsfonts,mathrsfs}
\usepackage{color}

\newcommand{\pmat}[1]{\begin{pmatrix} #1 \end{pmatrix}} %()
\newcommand{\dto}{\stackrel{{\rm d}}{\longrightarrow}}

\renewcommand{\tilde}{\widetilde}
\renewcommand{\bar}{\overline}

\newcommand{\bE}{\ensuremath{\mathbb{E}}}

\newcommand{\bN}{\ensuremath{\mathbb{N}}}

\newcommand{\bP}{\ensuremath{\mathbb{P}}}

\newcommand{\bR}{\ensuremath{\mathbb{R}}}

\newcommand{\cA}{\ensuremath{\mathcal{A}}}
\newcommand{\cB}{\ensuremath{\mathcal{B}}}

\newcommand{\cF}{\ensuremath{\mathcal{F}}}

\newcommand{\cL}{\ensuremath{\mathcal{L}}}

\newcommand{\cR}{\ensuremath{\mathcal{R}}}

\theoremstyle{plain}
\newtheorem{Thm}{Theorem}[section]

\newtheorem{Lem}[Thm]{Lemma}
\newtheorem{Prop}[Thm]{Proposition}
\newtheorem{Cor}[Thm]{Corollary}

\theoremstyle{definition}
\newtheorem{Ass}[Thm]{Assumption}
\newtheorem{Def}[Thm]{Definition}
\newtheorem{Rem}[Thm]{Remark}
\newtheorem{Ex}[Thm]{Example}

\newcommand{\Proof}[2][Proof]{\begin{proof}[{#1}] #2 \end{proof}}

\numberwithin{equation}{section}

\usepackage[dvipdfmx]{graphicx}
\begin{document}

\begin{center}
{\Large \bf 
Multiray generalization of the arcsine laws for occupation times of infinite ergodic transformations
}
\end{center}
\begin{center}
Toru Sera and Kouji Yano
\end{center}

\begin{abstract}
We prove that the joint distribution of the occupation time ratios for ergodic transformations preserving an infinite measure converges {\color{black}to a multidimensional version of Lamperti's generalized arcsine distribution, in the sense of strong distributional convergence}. 
Our results can be applied to interval maps and Markov chains.
We adopt the double Laplace transform method, which has been utilized in the study of occupation times of diffusions on multiray.
We also discuss the inverse problem. 
\end{abstract}

%%%%% text %%%%%

%%%%%%%%%%%%%%%%%%%%%%%%%%%%%%%%%%%%%%%%%%%%%%%%%%%%%%%%%%%%%%%%%
%%%%%%%%%%%%%%%%%%%%%%%%%%%%%%%%%%%%%%%%%%%%%%%%%%%%%%%%%%%%%%%%%
%%%%%%%%%%%%%%%%%%%%%%%%%%%%%%%%%%%%%%%%%%%%%%%%%%%%%%%%%%%%%%%%%
\section{Introduction}

Let $B=(B_t)_{t\geq0}$ be a one-dimensional standard Brownian motion with $B_0=0$. Then the occupation time ratio on the positive side for $B$  up to a fixed time $t>0$ is arcsine distributed, i.e., for any $y\in[0,1]$,  
\begin{align*}
  \bP\bigg[t^{-1}\int_0^t 1_{(0,\infty)}(B_s)ds\leq y\bigg]
  =
  \frac{2}{\pi}{\rm arcsin}\sqrt{y}.	
\end{align*}
In addition, we have a similar result in discrete time. Let $R=(R_k)_{k\geq0}$ be a one-dimensional simple symmetric random walk. Then the occupation time ratio on the positive side for $R$ up to $n-1$ converges in distribution to arcsine distribution as $n\to\infty$, i.e., for any $y\in[0,1]$,
\begin{align}\label{arcsine}
  \bP\bigg[n^{-1}\sum_{k=0}^{n-1} 1_{(0,\infty)}(R_k) \leq y\bigg]
  \to 
  \frac{2}{\pi}{\rm arcsin}\sqrt{y}.
  \tag{$\ast$}
\end{align}
These results are well-known as L\'evy's arcsine laws for occupation times (\S5 of \cite{Le}; see also Theorem III.4.2 of \cite{F1}). As a generalization of (\ref{arcsine}),  
Lamperti \cite{La} studied the class of discrete-time Markov processes $Z=(Z_k)_{k\geq0}$ on $\;A_1+\{0\}+A_2$, where the plus signs always mean the disjoint union, and where $A_1$ and $A_2$ will be called the \emph{rays}, having the following property: 
\emph{when $Z$ changes rays, it must visit the origin $0$}, i.e.,
the condition {\color{black}[}$Z_n \in A_i$ and $Z_m \in A_j$ for some $n<m$ and 
$i \neq j$ {\color{black}]} implies the existence of $n<k<m$ for which $Z_k = 0$. He obtained the sufficient and necessary condition of the existence of limit distributions of $n^{-1}\sum_{k=0}^{n-1}1_{A_1}(Z_k)$, and determined the class of possible limit distributions, which are called Lamperti's generalized arcsine distributions.
 
Partly inspired by \cite{La}, Thaler \cite{T02} proved 2-ray generalized arcsine laws for interval maps. We now explain a typical example of his result.  

\begin{Ex}[arcsine law for Boole's transformation (\cite{T02}, pp.1923--1924)]
Let $T:\bR\to\bR$ be defined by $Tx:=x-x^{-1}\;(x\neq0), \;T0:=0$. 
The map $T$ is called \emph{Boole's transformation}. It is known that $T$ is a conservative, ergodic and measure preserving transformation w.r.t.$\:$the Lebesgue measure on $\bR$ (this notion will be explained in Subsection \ref{subsec:notation}).   
Note that, \emph{when the orbit $(T^k x)_{k\geq 0}$  moves from  $(-\infty,-1)$ to $(1,\infty)$ or vice versa, it must visit $[-1,1]$}. So the state space $\bR$ can be decomposed as $\bR=(-\infty,-1)+[-1,1]+(1,\infty)$, which may be regarded as a $2$-ray: the two \emph{rays} $(-\infty,-1)$ and $(1,\infty)$ radiate from the \emph{junction} $[-1,1]$. 
We will denote by $S^+_n(x)$  the amount of times {\color{black}which} $(T^k x)_{k\geq0}$ {\color{black}spends} on $(1,\infty)$ up to $n-1$:
\begin{align*}
  S^+_n(x):=\sum_{k=0}^{n-1} 1_{(1,\infty)}(T^k x).	
\end{align*}
Then, for any probability measure $\nu$ absolutely continuous w.r.t.$\:$the Lebesgue measure on $\bR$, and for any $y\in[0,1]$, as $n\to\infty$,
\begin{align*}
 \nu\Big(x\in\bR\;;\;S^+_n(x)/n\leq y\Big) 
  \to 
  \frac{2}{\pi}{\rm arcsin}\sqrt{y},	
\end{align*}
i.e.,
\begin{align*}
	S^+_n/n \quad\text{{\color{black}(}under $\nu$\color{black}{)}}\; \dto
	\;\text{the arcsine distribution}, 
\end{align*}
where $\dto$ means the distributional convergence.
\end{Ex}

Thaler--Zweim\"uller \cite{TZ} {\color{black}and Zweim\"uller \cite{Z07}}  showed 2-ray generalized arcsine laws for infinite ergodic transformations as an extension of \cite{T02}.
In \cite{T02}, \cite{TZ} {\color{black}and \cite{Z07}}, their proofs were based on the moment method.

In this paper, we prove a joint distributional limit theorem as a multiray extension of {\color{black}\cite{Z07}}.
For this purpose, the moment method does not seem to be suitable. Instead, we adopt the double Laplace transform method, which was utilized in, e.g., Barlow--Pitman--Yor \cite{BPY}, Watanabe \cite{W95}, and Y. Yano \cite{Y17} for studies of generalized arcsine laws for diffusions.
These studies were based on the renewal property of excursions away from the origin, and so we cannot reduce our problem to these results because of lack of this property. 
We can nevertheless utilize the double Laplace transform of occupation times. We decompose it with respect to the length of the first excursion from the junction, use the stationarity and asymptotic recursions and finally show that the {\color{black}second leading} terms can be represented by the double Laplace transform of occupation times and the Laplace transforms of {\color{black}the differences of} wandering rates.
We now illustrate our main theorems by an example of interval maps.

\begin{Ex}[{\color{black}multidimensional} generalized arcsine law for interval map with indifferent fixed points]
We define the map $T:[0,1]\to[0,1]$ by
\begin{align*}
Tx:=
\begin{cases}
 x+c_1x^3,               &0\leq x \leq \frac{1}{3}, \\
 x+c_2(x-\frac{1}{2})^3, &\frac{1}{3}<x<\frac{2}{3}, \\
 x+c_3(x-1)^3,           &\frac{2}{3}\leq x\leq 1.
 \end{cases}
\end{align*}
We take $c_1=c_3=18$ and $c_2=72$ so that $T(\frac{i}{3}-)=1$ and $T(\frac{i}{3}+)=0$ for each $i$. 
\begin{figure}
\begin{center}
 \includegraphics[width=7cm]{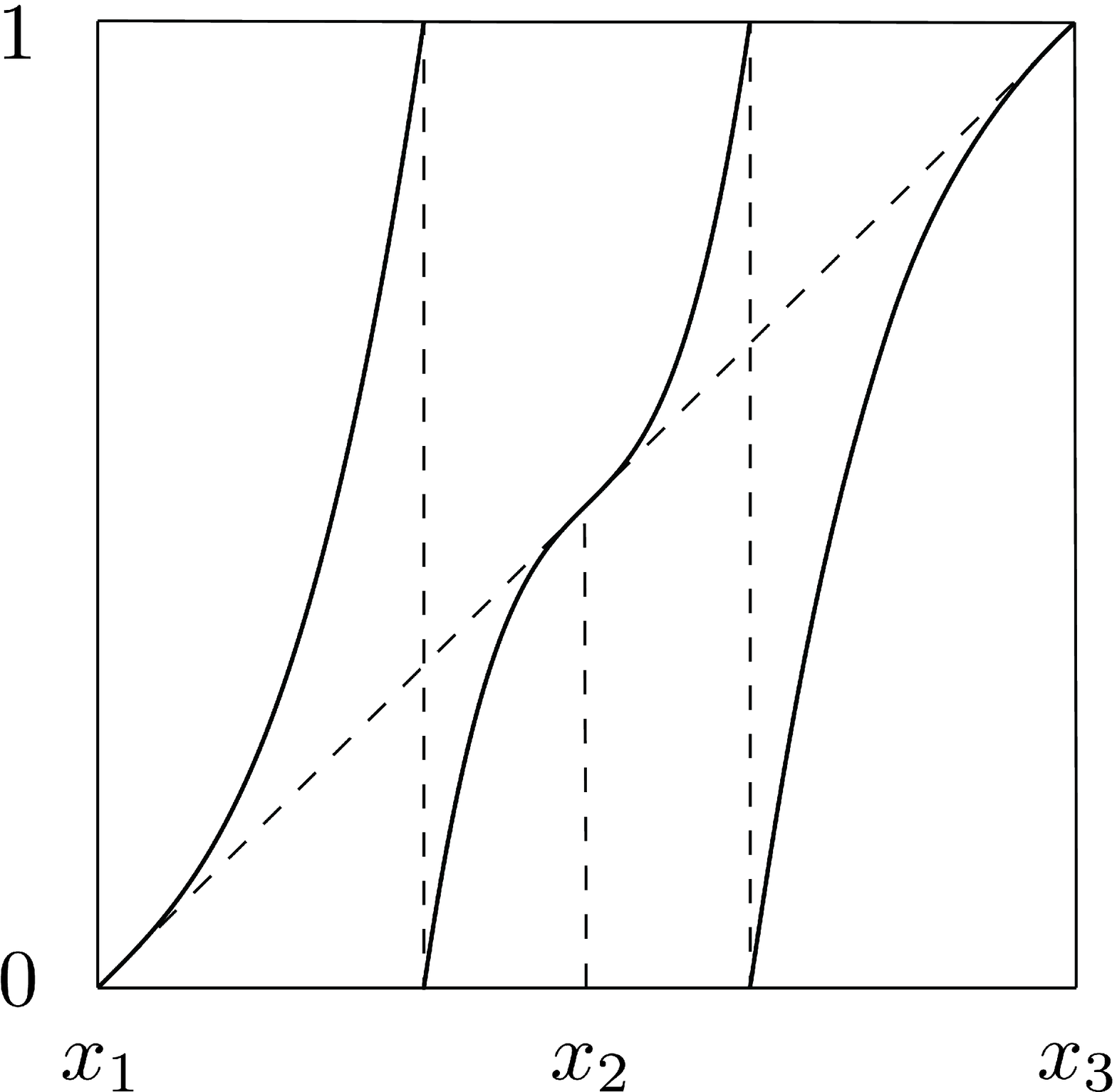}
\end{center}
\end{figure}
Note that $x_1=0,\:x_2=1/2,\:x_3=1$ are indifferent fixed points; in fact, for each $i=1,2,3$,
\begin{align*}
 	&\quad\quad\;\;\;
 	Tx_i=x_i, \quad T'x_i=1, \quad\text{and}
 	\\
 	&(x-x_i)T''x>0 \;\;\text{for any $x\neq x_i$ near $x_i$}.
 \end{align*}
Then $T$ has the unique (up to multiplication of positive
  constants) invariant measure $\mu$ on $[0,1]$ equivalent to the Lebesgue measure with $\mu([0,1])=\infty$,
   and $T$ is a conservative, ergodic, measure preserving transformation on $([0,1], \cB([0,1]),\mu)$ (see \cite{T80} and \cite{T83}).
Let $A_i$'s be disjoint small neighborhoods of $x_i$'s, respectively, and take $Y:=[0,1]\setminus \sum_{i=1}^3 A_i$. 
We will call $A_i$ the \emph{$i$-th ray} and $Y$ the \emph{junction}. We can take the rays sufficiently small so that, \emph{when the orbit $(T^k x)_{k\geq 0}$ changes rays, it must visit the junction}.  So $(T^k x)_{k\geq0}$ may be regarded as a process on $3$-ray $\sum_{i=1}^3 A_i+Y$. 
We will denote by $S^{(i)}_n (x)$ the amount of times {\color{black}which} $(T^k x)_{k\geq 0}$ {\color{black}spends} on the $i$-th ray up to $n-1$:
\begin{align*}
	S^{(i)}_n :=\sum_{k=0}^{n-1}1_{A_i}\circ T^k.
\end{align*}
Since $\mu(A_i)=\infty$ and $\mu(Y)<\infty$, by infinite ergodic theory (see Remark \ref{rem:birkoff}), 
\begin{align*}
 \sum_{i=1}^{3} S^{(i)}_n/n \to 1,
 \quad\text{$\mu$-a.e., as $n\to\infty$.}	
\end{align*}
We are interested in finding the limit of $(S^{(1)}_n,S^{(2)}_n,S^{(3)}_n)/n$ in a certain distributional sense. Our results show as follows: for any probability measure $\nu$ on $[0,1]$ absolutely continuous w.r.t.$\;\mu$ ({\color{black}equivalently}, w.r.t.$\;$the Lebesgue measure on $[0,1]$), 
\begin{align*}
  \frac{(S^{(1)}_n,S^{(2)}_n,
  S^{(3)}_n)}{n}
   \quad\text{{\color{black}(}under $\nu${\color{black})}}\;
   \dto\;
  \frac{(\xi_1,\xi_2,\xi_3)}{\sum_{i=1}^3\xi_i},
  \quad\text{as $n\to\infty$},
\end{align*}
where $\xi_1, \xi_2, \xi_3$ are $\color{black}[0,\infty)$-valued independent random variables with the one-sided $1/2$-stable distributions characterized by
\begin{align*}
   \bE\big[\exp(-\lambda\xi_i)\big] 
    = \exp(-\beta_i\sqrt{\lambda}), 
   \quad \lambda>0,\;i=1,2,3,
\end{align*}
for some constants $\beta_1,\beta_2,\beta_3>0$. The $\beta_i$'s are given by
\begin{align*}
\beta_i=
\begin{cases}
	c_i^{-1/2}\sum_{j\neq i}(h \circ f_j)(x_i)f_j'(x_i), &i=1,3,\\
	2c_i^{-1/2}\sum_{j\neq i}(h \circ f_j)(x_i)f_j'(x_i), &i=2,
\end{cases}
\end{align*}
where $h$ {\color{black}denotes} a lower semicontinuous version of $d\mu/dx$, and $f_i$ {\color{black}denotes} a $C^2$-extension over $[0,1]$ of the inverse of $T|_{(\frac{i-1}{3},\frac{i}{3})}$.
 For more details, see Subsection \ref{subsec:interval}.
\end{Ex}

The present paper is organized as follows. 
In Section \ref{sec:main}  we set up notation and terminology, state our general limit theorems, and give their applications to interval maps and Markov chains.
In Section \ref{sec:dob-lap}, we discuss several facts about the double Laplace transforms and strong convergence in distribution.
Section \ref{sec:integrating} is devoted to the study of  asymptotic recursions for ergodic transformations.
In Sections \ref{sec:pf} and \ref{sec:pf2}, we give the proofs to the direct and inverse limit theorems, respectively.
%{\color{black} In Appendix A, we recall some basic properties of the vague topology and Laplace transforms of locally finite measures on $[0,\infty)$.}

\subsection*{Acknowledgements}
The authors thank Takuma Akimoto for drawing their attention to the paper \cite{TZ}. They thank Christophe Profeta for his valuable suggestions on Subsection \ref{subsec:markov}.
They also thank the referee for his suggestion which helped improve Theorems \ref{thm:main} and \ref{thm:main2}.
The second author, Kouji Yano, was supported
by JSPS-MAEDI Sakura program and KAKENHI 26800058
and partially by KAKENHI 24540390 and KAKENHI 15H03624.

%%%%%%%%%%%%%%%%%%%%%%%%%%%%%%%%%%%%%%%%%%%%%%%%%%%%%%%%%%%%%%%%
%%%%%%%%%%%%%%%%%%%%%%%%%%%%%%%%%%%%%%%%%%%%%%%%%%%%%%%%%%%%%%%%
%%%%%%%%%%%%%%%%%%%%%%%%%%%%%%%%%%%%%%%%%%%%%%%%%%%%%%%%%%%%%%%%
 
\section{Main results}\label{sec:main}

\subsection{Notations {\color{black}and assumptions}}\label{subsec:notation}

Let $(X,\cA,\mu)$ be a {\color{black}$\sigma$-finite measure space} such that $\mu(X)=\infty$, 
and let $T:(X,\cA,\mu) \to (X,\cA,\mu)$ be a conservative, ergodic, measure preserving transformation, which is abbreviated {\color{black}as} \textit{CEMPT} (cf.$\;$\S1.0--1.3 of Aaronson \cite{A97}). 
We know from Proposition 1.2.2 of \cite{A97} that, if $T$ is a measure preserving transformation on $(X,\cA,\mu)$, then the condition  that $T$ is conservative and  ergodic is equivalent to the condition that for any $A\in \cA$ with $\mu(A)>0$,
\begin{align*}
	\sum_{n\geq0} 1_A \circ T^n
	= \infty, \quad\textrm{$\mu$-a.e.}
\end{align*}
The \textit{dual operator} 
$\widehat{T}:L^1(\mu) \to L^1(\mu)$ is characterized by 
$\int (\widehat{T}f)g \:d\mu= \int f(g \circ T)\:d\mu$ 
for any $f \in L^1(\mu)$ and $g \in L^\infty(\mu)$. 
Note that if $0\leq f \leq 1$ then $0 \leq \widehat{T}f \leq 1$, $\mu$-a.e.

We always denote the disjoint unions using plus signs; e.g., $A_1+A_2$, $\sum_i A_i$ and so on.
We will say that $Y\in\cA$ \textit{dynamically separates} $A_1, \dots, A_d\in\cA$ if the condition {\color{black}[}$x \in A_i$ and $T^n x\in A_j$ for some 
$i \neq j$ and  $\color{black}n\geq1${\color{black}]} implies the existence of some $k\in\{1,\dots,n-1\}$ for which $T^{k}x \in Y$.

{\color{black}
\begin{Ass}\label{ass:ray}
Let $d\geq2$ be a positive integer. Let $X$ be decomposed into $X = Y +\sum_{i=1}^d A_i$ for 
$Y\in\cA$ with $\mu(Y) \in (0,\infty)$ and $A_i\in \cA$ with $\mu(A_i)=\infty$ ($i=1,\dots,d$), such that $Y$ dynamically separates $A_1,\dots,A_d$.
\end{Ass}}

{\color{black}We assume Assumption \ref{ass:ray} from now on.}
We define 
\begin{align}
  \psi(x)&:=\min\{k\geq0\;;\;T^k x\in Y\}, \quad x\in X,
  \label{psi}
  \\
  Y_k &:=\{x\in X\;;\;\psi(x)=k\},\quad k\geq0.
  \label{Yk}
\end{align}
In particular, $Y_0 = Y$. Note that 
$\bigcup_{k=0}^{n-1}T^{-k}Y = \sum_{k=0}^{n-1}Y_k$ for $n\geq1$, and $X = \bigcup_{k\geq0}T^{-k}Y = \sum_{k\geq0}Y_k$, $\mu$-a.e.
{\color{black}
Let us denote by $(w(n))_{n\geq0}$ the wandering rate of $Y$, i.e., $w(0):=0$ and
\begin{align}
	w(n):= \mu\Bigg\{\bigcup_{k=0}^{n-1}
               (T^{-k}Y)\Bigg\} 
             = \sum_{k=0}^{n-1}\mu(Y_k),                         \quad n\geq1.
\end{align}
We will denote by $Q(s)$ the Laplace transform of $(w(n+1)-w(n))_{n\geq0}$:
\begin{align}\label{hatw}
    {\color{black}Q}(s) 
           :=\sum_{n\geq 0} e^{-ns} 
                   \Big\{w(n+1)-w(n)\Big\} 
            = \sum_{n\geq 0} e^{-ns}
                \mu(Y_n), 
                \quad s>0.
  \end{align}}\noindent
For each $i=1,\dots,d$, the \textit{wandering rate of $Y\:$starting from $\color{black}A_i$}, which will be denoted by $\color{black}(w_i(n))_{n\geq0}$, is defined by ${\color{black}w_i}(0):=0$ and
  \begin{align}
  \label{wA}
  {\color{black}w_i}(n) :=\mu\Bigg\{\bigcup_{k=0}^{n-1}
               (T^{-k}Y\cap {\color{black}A_i})\Bigg\} 
             = \sum_{k=0}^{n-1}\mu(Y_k \cap {\color{black}A_i}),                         \quad n\geq1,
  \end{align}
and the Laplace transform of $\color{black}(w_i(n+1)-w_i(n))_{n\geq 0}$ will be denoted by ${\color{black}Q_i}(s)$:
  \begin{align}\label{hatwA}
    {\color{black}Q_i}(s) 
           :=\sum_{n\geq 0} e^{-ns} 
                   \Big\{{\color{black}w_i}(n+1)-{\color{black}w_i}(n)\Big\} 
            = \sum_{n\geq 0} e^{-ns}
                \mu(Y_n\cap {\color{black}A_i}), 
                \quad s>0.
  \end{align}
It is easily seen that $0<{\color{black}Q_i}(s)<\infty$ for $s>0$ and $\color{black}Q_i(s) \to \infty$, as $\color{black}s\downarrow0$.

{\color{black}
\begin{Ass}\label{ass:comp} 
For each $i{\color{black}=1,\dots,d}$, the collection 
\begin{align}
	\mathfrak{H}_i:=
	\bigg\{
	\frac{1}{w_i(n)}
     \sum_{k=0}^{n-1}
      \widehat{T}^k 1_{Y_k \cap A_i}
      \;;\;n\geq1
	\bigg\}
\end{align}
is strongly precompact in $L^1(\mu)$.  
\end{Ass}}

%{\color{black}
%\begin{Rem}\label{supp}
%Note that
%\begin{align}\label{int:yk}
%   \mu(Y_k\cap A_i) = \int_{Y_k\cap A_i} 
%                  1_Y\circ T^k d\mu
%                  = \int_Y \widehat{T}^k 
%                       1_{Y_k\cap A_i}d\mu,
%   \quad k\geq 0,	
%\end{align}
%since $1 = 1_Y \circ T^k$ on $Y_k$.	Hence the $(w_i(n))^{-1}
%     \sum_{k=0}^{n-1}
%      \widehat{T}^k 1_{Y_k \cap A_i}$ is a probability density function w.r.t.$\:\mu$, and supported on $Y$. If $H_i\in L^\infty(\mu)$ satisfies
%  \begin{align}
%   \frac{1}{w_i(n)}
%     \sum_{k=0}^{n-1}
%      \widehat{T}^k 1_{Y_k \cap A_i }
%        \to H_i, \quad\text{in $L^\infty(\mu)$, 
%       as $n\to\infty$},
%  \end{align}
%then we will call $H_i$  the \textit{asymptotic entrance density of $Y\:$from $A_i$}. Note that for $g\in L^\infty(\mu)$, 
%  \begin{align*}
%  \frac{1}{\mu(\{\psi< n\}\cap A_i)}\int_{\{\psi< n\}\cap A_i}
%                g \circ T^{\psi}d\mu
%        &= \frac{1}{w_i(n)}\sum_{k=0}^{n-1} 
%                  \int_{Y_k \cap A_i} (g1_Y)\circ T^k d\mu \\       
%        &= \int_Y g\bigg( \frac{1}
%        {w_i(n)}
%        \sum_{k=0}^{n-1}
%        \widehat{T}^k
%         1_{Y_k \cap A_i}\bigg)d\mu \\             
%        &\to \int_Y gH_id\mu, 
%        \quad\text{as $n\to\infty$}.
%   \end{align*}
%It is easily seen that the existence of $H_i$ ensures the precompactness of $\mathfrak{H}_i$ in $L^1(\mu)$.
%\end{Rem}}

Let $f,g:(0,\infty)\to(0,\infty)$ be positive functions. {\color{black} For $c\in [0,\infty]$, we write $f(x)\sim cg(x)$, as $x\to x_0$, if it holds that $\lim_{x\to x_0}f(x)/g(x) =c$.
We note that $cg(x)$ has only a symbolic meaning if $c=0$ or $\infty$. See Bingham--Goldie--Teugels \cite{BGT}, p.$\:$xix.}

A positive and measurable function $f:(0,\infty)\to (0,\infty)$ is called regularly varying of index $\rho\in\bR$ at $\infty$, written $f\in\cR_\rho(\infty)$,  if, for each $r>0$, $f(rx)\sim r^\rho f(x)$, as $x\to\infty$.
A positive and measurable function $g$ is called regularly varying of index $\rho$ at $0$, written $g\in \cR_\rho(0+)$, if, for any each $r>0$, $g(rx)\sim r^\rho g(x)$, as $x\downarrow 0$. A positive sequence $(a_n)_{n\geq 0}$ is called regularly varying of index $\rho$ at $\infty$, written $(a_n)\in\cR_\rho(\infty)$, if the function $f(x):= a_{\lfloor x \rfloor}$ belongs to $\cR_\rho(\infty)$. Here $\lfloor x \rfloor$ means the greatest integer which is less than or equal to $x$. For a basic discussion of regularly varying functions, we refer the reader to Chapter 1 of \cite{BGT}.

We remark that $\sum_{i=1}^d w_i(n)=w(n)-\mu(Y)$ since $\sum_{i=1}^d A_i=\sum_{n\geq1}Y_n$.
Thus, $\sum_{i=1}^d w_i(n)\sim w(n)$, as $n\to\infty$.

\begin{Ass}\label{ass:reg}
For some constants $\alpha\in[0,1]$ and 
$\beta=(\beta_1,\dots,\beta_d) \in [0,1]^d$ with $\sum_{i=1}^d\beta_i=1$,
  \begin{align}
    &\quad\quad\quad\quad
    \big(w(n)\big)_{n\geq0} \in \cR_{1-\alpha}(\infty),
            \label{eq:reg}
   \\ 
    &
    w_i(n)\sim \beta_i w(n), \quad\text{as $n\to\infty$,  $i=1,\dots,d$.} \label{eq:bal}
\end{align}
\end{Ass}

\begin{Rem}\label{rem:reg-lap}
By Karamata's Tauberian theorem (see Proposition 4.2 and Remark 4.1 of \cite{TZ}), 
the set of the two conditions 
(\ref{eq:reg}) and (\ref{eq:bal}) is 
equivalent to that of the following two conditions:
  \begin{align}\label{eq:lap-reg}
   &\quad\quad\quad\quad
   {\color{black}Q} \in \cR_{-(1-\alpha)}(0+),
  \\
      \label{eq:lap-bal}
    &
    {\color{black}Q_i}(s)\sim \beta_i {\color{black}Q}(s), 
     \quad\text{as  $s\downarrow 0$,  $i=1,\dots, d$.}
  \end{align}
\end{Rem}

%%%%%%%%%%%%%%%%%%%%%%%%%%%%%%%%%%%%%%%%%%%%%%%%%%%%%%%%%%%%%%%%%
\subsection{Multidimensional generalized arcsine distributions}

Let $\xi_1,\xi_2$ be two ${\color{black}[0,\infty)}$-valued $\alpha$-stable random variables with {\color{black}Laplace transforms} given by
\begin{align*}
   \bE\big[\exp(-\lambda\xi_i)\big] 
    = \exp(-\beta_i\lambda^\alpha ), 
   \quad \lambda>0,\;i=1,2.	
\end{align*}
for $\alpha,\beta_1,\beta_2 \in(0,1)$ with $\beta_1+\beta_2=1$, and we assume  $\xi_1$ and $\xi_2$ are independent. We then recall that \emph{Lamperti's generalized arcsine distribution} \cite{La} (except degenerate cases) is
 the distribution of $\xi_1/(\xi_1+\xi_2)$ given by 
\begin{align*}
  \bP\bigg[\frac{\xi_1}{\xi_1+\xi_2}\leq y\bigg]
  &= 
  \frac{\sin (\pi \alpha)}{\pi} 
  \int_0^y 
   \frac{\beta_1\beta_2 x^{-(1-\alpha)}(1-x)^{-(1-\alpha)}dx}
        {\beta_1^2(1-x)^{2\alpha}+\beta_2^2 x^{2\alpha}
          +2\beta_1\beta_2 x^\alpha (1-x)^\alpha \cos(\pi \alpha)}
  \\
  &=
  \frac{1}{\pi\alpha}{\rm arccot} 
   \bigg[\frac{\beta_1 (1-y)^\alpha}
              {\beta_2 y^\alpha \sin (\pi \alpha)}
         + {\rm cot} (\pi \alpha)\bigg]
  ,
\quad y\in(0,1].	
\end{align*}
In the special case of $\alpha=\beta_1=\beta_2=1/2$, this distribution is nothing else but the usual arcsine distribution:
\begin{align*}
 \bP\bigg[\frac{\xi_1}{\xi_1+\xi_2}\leq y\bigg]
 =\frac{1}{\pi}\int_0^y \frac{dx}{\sqrt{x(1-x)}}
 =\frac{2}{\pi}{\rm arcsin}\sqrt{y},
 \quad y\in[0,1].
\end{align*}
For more details, see, e.g., \S2--3 of \cite{T02}.
We now recall its multidimensional generalization including degenerate cases. Let $d\geq2$ be a positive integer.

\begin{Def}[multidimensional generalized arcsine distributions \cite{BPY}, \cite{Y17}]\label{def:arcsine}
 For $\alpha\in[0,1]$ and $\beta = (\beta_1,\dots,\beta_d) \in [0,1]^d$ with $\sum_{i=1}^d \beta_i=1$, we write $\zeta_{\alpha,\beta}$ for a $[0,1]^d$-valued random variable whose distribution is characterized as follows:
\begin{enumerate}
\item[(1)] if $0<\alpha<1$, the $\zeta_{\alpha,\beta}$ is equal in distribution to 
 \begin{equation}
      \displaystyle
         \left( \frac{\xi_1}{\sum_{i=1}^d\xi_i}, \dots, 
         \frac{\xi_d}{\sum_{i=1}^d\xi_i}\right),
 \end{equation}
where $\xi_1, \dots, \xi_d$ {\color{black}denote} ${\color{black}[0,\infty)}$-valued independent random variables with the one-sided $\alpha$-stable distributions characterized by
  \begin{align}
   \bE\big[\exp(-\lambda\xi_i)\big] 
    = \exp(-\beta_i\lambda^\alpha), 
   \quad \lambda >0,\;i=1,\dots,d.
  \end{align}
\item[(2)] if $\alpha=1$, the $\zeta_{1,\beta}$ is equal a.s.$\;$to the constant $\beta$.

\item[(3)] if $\alpha=0$, the distribution of $\zeta_{0,\beta}$ is $\sum_{i=1}^d \beta_i \delta_{e^{(i)}}$ with $e^{(i)}=(1_{\{i=j\}})_{j=1}^d \in [0,1]^d$ for $i=1,\dots,d$.
\end{enumerate}
\end{Def}

We call $\zeta_{\alpha,\beta}$ {\em trivial} if $\beta=e^{(i)}$ for some $i$. In this case we have $\zeta_{\alpha,e^{(i)}}=e^{(i)}$, a.s.,$\:$whatever $\alpha$ is.

For $a=(a_1,\dots,a_d)$ and $b=(b_1,\dots,b_d)\in\bR^d$, 
we write $a\cdot b :=\sum_{i=1}^d a_i b_i$ for the inner product of $a$ and $b$.
The distribution of $\zeta_{\alpha,\beta}$ is characterized by the following formula of double Laplace transformation.

\begin{Prop}\label{prop:dob-lap-arc}
Let $\zeta$ be a $\color{black}[0,\infty)^d$-valued random variable. Then $\zeta$ is equal in distribution to $\zeta_{\alpha,\beta}$ if and only if for any $q>0$ and $\lambda=(\lambda_1,\dots,\lambda_d)\in {\color{black}[0,\infty)^d}$, 
\begin{align}\label{eq:dob-lap-arc}
\int_0^\infty du \:e^{-qu}
\bE\big[\exp(-u\lambda\cdot \zeta)\big]
=  \frac{\sum_{i=1}^d  \beta_i(q+\lambda_i)^{-(1-\alpha)}}
          {\sum_{i=1}^d  \beta_i(q+\lambda_i)^\alpha }.
\end{align} 	
\end{Prop}

For the proof, see, e.g., the proof of Proposition 3.6 of Y. Yano \cite{Y17}.

%%%%%%%%%%%%%%%%%%%%%%%%%%%%%%%%%%%%%%%%%%%%%%%%%%%%%%%%%%%%%%%%%

\subsection{Limit theorems in a general setting}

Let $Z$ be a Polish space, $(F_n)$  a sequence of $Z$-valued measurable functions defined on $(X,\cA)$, 
and $\zeta$  a $Z$-valued random variable defined on a probability space $(\Omega, \cF, \bP)$. 
If for a probability measure $\nu$ on $(X,\cA)$, the distribution of $F_n$ under $\nu$ converges to that of $\zeta$ under $\bP$, 
then we write $F_n \overset{\nu}{\Longrightarrow} \zeta$. 
We say that $F_n$ \textit{converges to $\zeta$ strongly in distribution} if for any probability measure $\nu \ll \mu$ on $(X,\cA)$, it holds that $F_n \overset{\nu}{\Longrightarrow} \zeta$.
We denote this convergence by $F_n \overset{\cL(\mu)}{\Longrightarrow} \zeta$.

Set $S_0:={\color{black}(0,\dots,0)}$ and
\begin{align}\label{occupation}
S_n :=\bigg(\sum_{k=1}^n 1_{A_1}\circ T^k, \dots, 
            \sum_{k=1}^n 1_{A_d}\circ T^k \bigg),
            \quad{n\geq1}.
\end{align}
The following theorem for $d=2$ was due to {\color{black}Zweim\"uller (Theorem 2.2 of \cite{Z07})}.

\begin{Thm}[direct limit theorem]\label{thm:main}
Let $T$ be a CEMPT on $(X,\cA,\mu)$ and suppose that Assumptions \ref{ass:ray}, {\color{black}\ref{ass:comp}} and \ref{ass:reg} hold.
Then
 \begin{align*}
  S_n/n
     \overset{\cL(\mu)}{\Longrightarrow}
     \zeta_{\alpha,\beta},
     \quad\text{as $n\to\infty$}.
  \end{align*}
\end{Thm}

The proof of Theorem \ref{thm:main} will be given in Section \ref{sec:pf}.

Conversely, under Assumptions \ref{ass:ray} {\color{black}and \ref{ass:comp}}, if $S_n/n$ converges in distribution under some probability measure, then Assumption \ref{ass:reg} must hold except certain trivial cases. In particular, the class of possible limit distributions of $S_n/n$ coincides the set of distributions of $\{\zeta_{\alpha,\beta}\}_{\alpha,\beta}$. 

\begin{Thm}[inverse limit theorem] \label{thm:main2}
Let $T$ be a CEMPT on $(X,\cA,\mu)$ and suppose that Assumptions \ref{ass:ray} {\color{black} and \ref{ass:comp}} hold. Furthermore, suppose that there exist a probability measure $\nu_0 \ll \mu$ and a $[0,1]^d$-valued random variable $\zeta$ such that 
\begin{align*}
	S_n/n \overset{\nu_0}{\Longrightarrow} \zeta,
	\quad\text{as $n\to\infty$.}
\end{align*}
Then $\zeta$ is equal in distribution to $\zeta_{\alpha,\beta}$ for some $\alpha\in[0,1]$ and $\beta =(\beta_1,\dots,\beta_d)\in[0,1]^d$ with $\sum_{i=1}^d\beta_i =1$,
and 
\begin{align*}
   S_n/n \overset{\cL(\mu)}{\Longrightarrow} \zeta_{\alpha,\beta},
   \quad\text{as $n\to\infty$}.
\end{align*}
Moreover, if $\zeta_{\alpha,\beta}$ is not trivial, then the two conditions {\rm (\ref{eq:reg})} and {\rm (\ref{eq:bal})} hold.
\end{Thm}

We will give the proof of Theorem \ref{thm:main2} in Section \ref{sec:pf2}.

\begin{Cor}\label{cor:excur}
Let $\alpha\in[0,1)$ and $\beta=(\beta_1,\dots,\beta_d) \in [0,1]^d$ with $\sum_{i=1}^d\beta_i=1$. Suppose that $T$ is a CEMPT on $(X,\cA,\mu)$ and Assumptions \ref{ass:ray} {\color{black} and \ref{ass:comp}} hold. We consider the following conditions:
\begin{itemize}
 \item[{\rm (i)}] There exists $\Phi\in\cR_{-\alpha}(\infty)$ such that
\begin{align*}
   \mu[x\in Y\:;\:Tx,\dots,T^n x\in A_i]
    \sim
    \beta_i \Phi(n), 
    \quad\text{as $n\to\infty$},\; i=1,\dots,d. 	
\end{align*}
 \item[{\rm (ii)}] $S_n/n \overset{\cL(\mu)}{\Longrightarrow} \zeta_{\alpha,\beta},
   \;\text{as $n\to\infty$}$.
\end{itemize}
Then, {\rm (i)} implies {\rm (ii)}. Furthermore, if $\zeta_{\alpha,\beta}$ is not trivial, then {\rm (ii)} implies {\rm (i)}.  
\end{Cor}

\begin{Rem}
As we shall see in Lemma \ref{lem:excur},
 \begin{align*}
  w_i(n+1)-w_i(n)=\mu[x\in Y\:;\:Tx,\dots,T^n x\in A_i],
  \quad n\geq1,\;i=1,\dots,d.	
 \end{align*}
Hence, by Karamata's Tauberian theorem, if $\alpha \in [0,1)$ and $\beta=(\beta_1,\dots,\beta_d) \in [0,1]^d$ with $\sum_{i=1}^d \beta_i=1$, then the condition (i) of Corollary \ref{cor:excur} is equivalent to the set of the two conditions (\ref{eq:reg}) and (\ref{eq:bal}). 
\end{Rem}

\begin{Rem}\label{rem:birkoff}
Because of $\mu(X)=\infty$, for any $B\in\cA$ with $\mu(B)<\infty$, 
\begin{align*}
n^{-1}\sum_{k=1}^n 1_{B}\circ T^k \to 0,
\quad\text{$\mu$-a.e., as $n\to\infty$}. 
\end{align*}
This convergence is easily deduced from Hopf's ratio ergodic theorem. See, e.g., Theorem 2.2.5 and Exercise 2.2.1 of \cite{A97}. Hence Theorems \ref{thm:main} and \ref{thm:main2} and Corollary \ref{cor:excur} remain valid if we replace $S_n$ with 
\begin{align*}
S'_n:=\bigg(\sum_{k=1}^n 1_{A'_1}\circ T^k, \dots, 
            \sum_{k=1}^n 1_{A'_d}\circ T^k \bigg),
\end{align*}
where $A'_i\in\cA$ satisfies $\mu(A_i \triangle A'_i)<\infty$ for $i=1,\dots,d$. 
\end{Rem}

%%%%%%%%%%%%%%%%%%%%%%%%%%%%%%%%%%%%%%%%%%%%%%%%%%%%%%%%%%%%%%%%%
\subsection{Application to interval maps with indifferent fixed points}
\label{subsec:interval}

Let $d\geq2$ be a positive integer, and let $0=a_0=x_1<a_1<x_2<\dots<a_{d-1}<x_d=a_d=1$. Suppose that the map $T:[0,1]\to[0,1]$ satisfies the following conditions: for each $i=1,\dots,d$,
\begin{enumerate}
 \item[(1)]
 the restriction $T|_{(a_{i-1},a_i)}$ has a $C^2$-extension over $[a_{i-1},a_i]$, and $\bar{T\big((a_{i-1},a_i)\big)}=[0,1]$,
 \item[(2)]
 the $x_i$ satisfies
 \begin{align*}
 	&\quad\quad\quad\;\;
 	Tx_i=x_i, \quad T'x_i=1, \quad\text{and}
 	\\
 	&(x-x_i)T''x>0 \;\;\text{for any $x\in (a_{i-1},a_i)\setminus\{x_i\}$}.
 \end{align*}
 In particular, $T'>1$ on $(a_{i-1},a_i)\setminus\{x_i\}$ and consequently $x_i$ is an indifferent fixed point of $T$.
\end{enumerate}
Then $T$ has the unique (up to multiplication of positive constants) $\sigma$-finite invariant measure $\mu$ equivalent to the Lebesgue measure and $T$ is a CEMPT on $\big([0,1],\cB([0,1]),\mu\big)$ (see Section 1 of \cite{T83}). By the assumptions (1) and (2), we see that each $T|_{(a_{i-1},a_i)}$ is invertible, and its inverse has a $C^2$-extension over $[0,1]$, which will be denoted by $f_i$. We also see that 
\begin{align*}
  |x-f_i (x)|=O(|x-x_i|^2), \quad\text{as $x\to x_i$}.	
\end{align*}
We know from Lemma 4 of \cite{T83} that the density of $\mu$ w.r.t.$\:$the Lebesgue measure has a version $h$ of the form  
\begin{align}
	h(x)=h_0(x)\prod_{i=1}^d
	     \frac{x-x_i}{x-f_i (x)},
	\end{align}
where $h_0$ is continuous and positive on $[0,1]$. Thus any neighborhood of $x_i$ has infinite volume w.r.t.$\:\mu$.

We can take $\varepsilon>0$ sufficiently small so that the sets
\begin{align*}
   A_i^+&:= (x_i, x_i+\varepsilon),
   \quad i=1,\dots,d-1,
   \\
   A_i^-&:= (x_i-\varepsilon,x_i),
   \quad i=2,\dots,d,
\end{align*}
satisfy $T(A_i^\pm)\subset (a_{i-1},a_i)$.
Thus the $A_i^\pm$'s are disjoint and $\mu(A_i^\pm)=\infty$. Set for $n\geq1$
\begin{align*}
  &S_n^{(i,\pm)} :=\sum_{k=1}^n 1_{A_i^\pm}\circ T^k,
  \\
  &S_n := \big( S_n^{(i,\pm)}\big)_{i,\pm} =\big(S_n^{(1,+)},S_n^{(2,-)},S_n^{(2,+)},\dots,S_n^{(d,-)}\big).
\end{align*}
The following corollaries are multiray extensions of the 2-ray results \cite{T02}. Recall that, if $\lim_{x\to x_0} f(x)/g(x)=c\in[0,\infty]$, we write $f(x)\sim cg(x)$, as $x\to x_0$.

\begin{Cor}[direct limit theorem]\label{cor:indiff-dir}
Let $\alpha\in(0,1)$. Suppose that there exist $\Psi\in\cR_{1+1/\alpha}(0+)$ and  $c=(c_i^\pm)_{i,\pm}\in(0,\infty]^{2d-2}\setminus\{\infty\}^{2d-2}$, such that, for each $i$ and $\pm$,
 \begin{align}\label{reg:T}
 	|Tx-x|\sim c_i^\pm \Psi\big(|x-x_i|\big),
 	\quad\text{as $x\to x_i\pm0 $}.
 \end{align}
Then
\begin{align*}
S_n/n \overset{\cL(\mu)}{\Longrightarrow} \zeta_{\alpha,\beta},
\end{align*}
for the constant $\beta=(\beta_i^\pm)_{i,\pm}\in[0,1]^{2d-2}$ with $\sum_{i,\pm}\beta_i^\pm=1$, defined by
\begin{align}\label{betaipm}
 \beta_i^\pm :=\frac{(c_i^\pm)^{-\alpha} v_i} {\sum_{j,\pm}(c_j^\pm)^{-\alpha} v_j},
 \quad\text{with }
 v_i:= \sum_{j\neq i} (h\circ f_j) (x_i) f'_j(x_i).
\end{align}
\end{Cor}

\begin{Cor}[inverse limit theorem]\label{cor:indiff-inv}
Suppose that there exist a probability measure $\nu \ll \mu$ and a $[0,1]^{2d-2}$-valued random variable $\zeta$ such that
\begin{align*} 
S_n/n \overset{\nu}{\Longrightarrow} \zeta.
\end{align*}
Then $\zeta$ is equal in distribution to $\zeta_{\alpha,\beta}$ for some $\alpha\in[0,1]$ and $\beta =(\beta_i^\pm)_{i,\pm}\in[0,1]^{2d-2}$ with $\sum_{i,\pm}\beta_i^\pm =1$,
and 
\begin{align*}
S_n/n \overset{\cL(\mu)}{\Longrightarrow} \zeta_{\alpha,\beta}.
\end{align*}
Forthermore, if $\alpha\in(0,1)$ and $\beta\in[0,1)^{2d-2}$, then there exist $\Psi\in\cR_{1+1/\alpha}(0+)$ and $c=(c_i^\pm)_{i,\pm}\in(0,\infty]^{2d-2}$, such that, at least two of the $c_i^\pm$'s are finite, and the asymptotic  behaviors $(\ref{reg:T})$ hold for all $i$ and $\pm$.
\end{Cor}

Corollaries \ref{cor:indiff-dir} and \ref{cor:indiff-inv} can be deduced from Theorems \ref{thm:main} and \ref{thm:main2} and several facts in the earlier studies, e.g., \cite{T83}, \cite{T02} and \cite{TZ}. For the convenience of the reader, we give the proof.

Set
\begin{align*}
  Y&:=[0,1]\setminus \sum_{i,\pm}A_i^\pm.
\end{align*}
Then $\mu(Y)\in(0,\infty)$ and $Y$ dynamically separates the $A_i^\pm$'s.

\begin{Prop}\label{prop:interval-map}
{\rm (1)} For each $i$ and $\pm$, there exists {\color{black} bounded continuous function $H_i^\pm: Y\to [0,\infty)$ such that 
\begin{align}
\frac{1}{w_i^{\pm}(n)}
     \sum_{k=0}^{n-1}
      \widehat{T}^k 1_{Y_k \cap A_i^\pm }
        \to H_i^\pm, \quad\text{in $L^\infty(\mu|_Y)$, 
       as $n\to\infty$,}	
\end{align}
where $(w_i^{\pm}(n))_{n\geq0}$ denotes the wandering rate of $Y$ starting from $A_i^\pm$, and $Y_k$ is given by {\rm (\ref{Yk})}.}

{\rm (2)} For each $i$, as $n\to\infty$,
 \begin{align}
	\mu[x\in Y\:;\:Tx,\dots,T^n x\in A_i^+]
     &\sim
	v_i\Big(f_i^{n}(1)-x_i\Big),
	\\
	\mu[x\in Y\:;\:Tx,\dots,T^n x\in A_i^-]
     &\sim
	v_i\Big(x_i-f_i^{n}(0)\Big),
 \end{align}
for the constant $v_i$ defined by {\rm (\ref{betaipm})}.
\end{Prop}

Proposition \ref{prop:interval-map} can be easily obtained from the proof of Theorem 8.1 of \cite{TZ}. So we omit its proof.  
By Proposition \ref{prop:interval-map}, the condition of Assumption {\color{black}\ref{ass:comp}} is satisfied.
Set 
\begin{align}
 U_i^+(x)&:=\int_{x_i+x}^{1}\frac{dy}{y-f_i(y)}, \quad x\in(0,1-x_i],
 \\
 U_i^-(x)&:=\int_0^{x_i-x} \frac{dy}{f_i(y)-y}, \quad x\in(0,x_i].
\end{align}
By Lemma 2 of \cite{T83}, we have
\begin{align*}
 U_i^+\big(f_i^{n}(1)-x_i\big)\sim n,
 \quad\text{as $n\to\infty$}.
\end{align*}
Since $U_i^+$ and $\big(f_i^n(1)-x_i\big)_{n\geq0}$ are strictly decreasing, we obtain
\begin{align}
	f_i^{n}(1)-x_i\sim (U_i^+)^{-1}(n), \quad\text{ as $n\to\infty$},
\end{align}
where $(U_i^+)^{-1}$ denotes the inverse function of $U_i^+$. Similarly, 
\begin{align}
	x_i-f_i^n(0)\sim (U_i^-)^{-1}(n), \quad\text{ as $n\to\infty$}.
\end{align} 
Let $\Psi$ and the $c_i^\pm$ be as in the assumption of Corollary \ref{cor:indiff-dir}. We remark that, for each $i$ and $\pm$, the asymptotic behavior (\ref{reg:T}) is equivalent to
\begin{align}
	|x-f_i(x)|\sim c_i^\pm \Psi\big(|x-x_i|\big),
 	\quad\text{as $x\to x_i\pm0 $}.
\end{align}
By the above arguments and by some basic facts of regular variation, 
we has obtained the following proposition:

\begin{Prop}\label{prop:equivalence-reg}
Let $\alpha\in(0,1)$, $c=(c_i^\pm)_{i,\pm}\in(0,\infty]^{2d-2}\setminus\{\infty\}^{2d-2}$, and let $\beta=(\beta_i^\pm)_{i,\pm}\in [0,1]^{2d-2}$ be defined by $(\ref{betaipm})$. Then the following are equivalent:
\begin{enumerate}
\item[\rm (i)] There exists $\Psi\in\cR_{1+1/\alpha}(0+)$ such that $(\ref{reg:T})$ hold for all $i$ and $\pm$.

\item[\rm (ii)] There exists $\Phi\in\cR_{-\alpha}(\infty)$ such that, for each $i$ and $\pm$,
\begin{align*}
	\mu[x\in Y\:;\:Tx,\dots,T^n x\in A_i^\pm]
    \sim
    \beta_i^\pm \Phi(n), 
    \quad\text{as $n\to\infty$}.
\end{align*}
\end{enumerate}
\end{Prop}

Combining Propositions \ref{prop:interval-map} and \ref{prop:equivalence-reg}, Theorem \ref{thm:main2} and Corollary \ref{cor:excur}, we obtain Corollaries \ref{cor:indiff-dir} and \ref{cor:indiff-inv}.

%%%%%%%%%%%%%%%%%%%%%%%%%%%%%%%%%%%%%%%%%%%%%%%%%%%%%%%%%%%%%%%%%
\subsection{Application to Markov chains on multiray}\label{subsec:markov}

Let $Z=(Z_k)_{k\geq0}$ be an irreducible and null-recurrent discrete-time Markov chain with an invariant measure $\tilde{\mu}$ on a countable discrete state space $\widetilde{X}=\{0\}+\sum_{i=1}^d \widetilde{A}_i$, where $\widetilde{A}_1,\dots,\widetilde{A}_d$ will be called the \emph{rays}, having the following property: 
$Z$ cannot skip the origin $0$ when it changes rays, i.e.,
the condition $Z_n \in \widetilde{A}_i$ and $Z_m \in \widetilde{A}_j$ for some $n<m$ and 
$i \neq j$ implies the existence of $n<k<m$ for which $Z_k = 0$.
We denote by $\bP_{\tilde{x}}$ the law of $Z$ with $Z_0=\tilde{x}\in \widetilde{X}$.
Set, for a probability measure $\tilde{\nu}$ on $\widetilde{X}$, and for $n\geq1$,
\begin{align*}
 \bP_{\tilde{\nu}} &:= \int_{\widetilde{X}} \bP_{\tilde{x}}(\cdot) \tilde{\nu} (d\tilde{x}),
 \\
 S_n &:=\bigg(\sum_{k=1}^n 1_{\widetilde{A}_1}(Z_k),\dots,\sum_{k=1}^n 1_{\widetilde{A}_d}(Z_k)\bigg).	
\end{align*}

By our theorems, we obtain the following corollary, which is a multiray extension of the $2$-ray result \cite{La}.

\begin{Cor}\label{markov}
Let $\alpha\in[0,1)$ and $\beta=(\beta_1,\dots,\beta_d)\in[0,1]^d$ with $\sum_{i=1}^d \beta_i=1$. We consider the following conditions:
\begin{itemize}
 \item[{\rm (i)}] There exists $\Phi\in\cR_{-\alpha}(\infty)$ such that
\begin{align*}
   \bP_0[Z_1,\dots,Z_n \in \widetilde{A}_i]
    \sim
    \beta_i \Phi(n), 
    \quad\text{as $n\to\infty$},\; i=1,\dots,d. 	
\end{align*}
 \item[{\rm (ii)}] For any probability measure $\tilde{\nu}$ on $\widetilde{X}$, the $S_n/n$ under $\bP_{\tilde{\nu}}$ converges in distribution to $\zeta_{\alpha,\beta}$, as $n\to\infty$.
\end{itemize}
Then, {\rm (i)} implies {\rm (ii)}. Furthermore, if $\zeta_{\alpha,\beta}$ is not trivial, then {\rm (ii)} implies {\rm (i)}. 	
\end{Cor}

Let us prove Corollary \ref{markov} using Corollary \ref{cor:excur}.  Without loss of generality, we may assume that $\tilde{\mu}(\{0\})=1.$ Set
\begin{align*}
  X&:=\widetilde{X}^{\bN}=\{x=(x_k)_{k\geq0}\:;\:x_0,x_1,\ldots \in \widetilde{X}\},\\
  Y&:=\{x\in X\:;\:x_0=0\},\\
A_i&:=\{x\in X\:;\:x_0\in \widetilde{A}_i\},
\quad i=1,\dots,d.
\\	
  \mu &:=\bP_{\tilde{\mu}}
      =\int_{\widetilde{X}} 
        \bP_{\tilde{x}}(\cdot)\tilde{\mu}(d\tilde{x}).
\end{align*}
We define the shift operator $T:X\to X$ by $T(x_k)_{k\geq0}=(x_{k+1})_{k\geq0}$.

By Theorem 4.5.3 of \cite{A97}, the shift operator $T$ is a CEMPT on $(X,\cB(X),\mu)$.
 Furthermore, all the conditions of Assumptions \ref{ass:ray} and {\color{black}\ref{ass:comp}} are satisfied. In fact, we define $\psi,Y_k$ and $w_i$ by (\ref{psi}), (\ref{Yk}) and (\ref{wA}), respectively. Then, for any $g\in L^{\infty}(\mu)$, 
\begin{align*}
  \int_X g\bigg( \frac{1}{w_i(n)} \sum_{k=0}^{n-1}
        \widehat{T}^k 1_{Y_k \cap A_i}\bigg)d\mu
   &=
   \frac{1}{w_i(n)}\int_{\bigcup_{k=0}^{n-1}(T^{-k}Y\cap A_i)}
                g \circ T^{\psi}d\mu   
   \\
   &=
   \frac{1}{w_i(n)}
   \int_{\widetilde{A}_i}\bE_{\tilde{x}}
    \big[g(T^{\psi}Z)\:;\:\psi<n\big]
    \:\tilde{\mu}(d\tilde{x})
   \\
   &=
   \frac{1}{w_i(n)}
   \int_{\widetilde{A}_i}\bE_{0}\big[g(Z)\big]
    \bP_{\tilde{x}}\big[\psi<n\big]
    \:\tilde{\mu}(d\tilde{x})
   \\
   &=
   \frac{1}{w_i(n)}
   \bE_{0}\big[g(Z)\big]
    \mu(x\in A_i\:;\:\psi(x)<n)
   \\
   &=
   \bE_0\big[g(Z)\big]
   = 
   \int_{Y} g d\mu.
\end{align*}
Here we used the strong Markov property.
Hence 
\begin{align*}
    \frac{1}{w_i(n)} \sum_{k=0}^{n-1}
           \widehat{T}^k 1_{Y_k \cap A_i}=1_Y, \quad\text{$\mu$-a.e., $n\geq1$,}
\end{align*} 
{\color{black}and the condition of Assumption \ref{ass:comp} is satisfied.}
The conditions of Assumption \ref{ass:ray} are trivially satisfied.

Note that, for each probability measure $\tilde{\nu}$ on $\widetilde{X}$, the probability measure $\bP_{\tilde{\nu}}$ is absolutely continuous w.r.t.$\:\mu$. Hence, by Corollary \ref{cor:excur}, we obtain Corollary \ref{markov}.

%%%%%%%%%%%%%%%%%%%%%%%%%%%%%%%%%%%%%%%%%%%%%%%%%%%%%%%%%%%%%%%%%
%%%%%%%%%%%%%%%%%%%%%%%%%%%%%%%%%%%%%%%%%%%%%%%%%%%%%%%%%%%%%%%%%
%%%%%%%%%%%%%%%%%%%%%%%%%%%%%%%%%%%%%%%%%%%%%%%%%%%%%%%%%%%%%%%%%
\section{Convergence of double Laplace transforms implying strong convergence in distribution}\label{sec:dob-lap}

Recall that, for $a=(a_1,\dots,a_d)$ and $b=(b_1,\dots,b_d)\in\bR^d$, 
we write $a\cdot b :=\sum_{i=1}^d a_i b_i$.
Suppose that $T$ is a CEMPT on $(X,\cA,\mu)$, and that Assumption \ref{ass:ray} holds. Write 
\begin{align}\label{tildeS}
\widetilde{S}_{\lambda,n,t}:=\exp(-\lambda\cdot S_n/t),
\quad \lambda\in{\color{black}[0,\infty)^d},\;n\geq0, \;t>0,
\end{align}
{\color{black}where $S_n$ has been defined by (\ref{occupation}).}
 In the following, $\lfloor t \rfloor$ means the greatest integer that is less than or equal to $t$.
\begin{Lem}\label{lem:main-a}
Suppose that $T$ is a CEMPT on $(X,\cA,\mu)$, and that Assumption \ref{ass:ray} holds. Let $\zeta$ be a $[0,\infty)^d$-valued random variable, and $I\subset(0,\infty)$ a {\color{black}non-empty open} interval. 
Assume that for any $q\in I$, $\lambda\in[0,\infty)^d$  
and for any probability measure  
$\nu\ll\mu$ on $(X,\cA)$, as $t\to\infty$,
  \begin{align*}
    \int_0^\infty du\;e^{-qu}
    \int_X d\nu\;\widetilde{S}_{\lambda,\lfloor ut \rfloor,t}
     \to 
     \int_0^\infty du\;e^{-qu}
       \bE\left[\exp(-u\lambda\cdot \zeta)\right]. 
  \end{align*}
Then $S_n/n \overset{\cL(\mu)}{\Longrightarrow} \zeta$.
\end{Lem}

\begin{Lem}\label{lem:lap}
For $n \in \bN\cup\{\infty\}$, let $f_n :(0,\infty)\to[0,\infty)$ 
be a non-increasing function. Assume that there exists a {\color{black}non-empty open} interval $I\subset(0,\infty)$ such that for any $q\in I$, as $n\to\infty$,
  \begin{align*}
    \int_0^\infty e^{-qu}f_n(u)du \to 
     \int_0^\infty e^{-qu}f_\infty (u)du < \infty.
  \end{align*}
Then $f_n(u) \to f_\infty (u)$ for all continuity points $\color{black}u\in(0,\infty)$ of $f_\infty$.
\end{Lem}

\Proof{
By the continuity theorem (see Theorem XIII.1.2a of \cite{F2}), we have, for any $0<a,b<\infty$, as $n\to\infty$,
\begin{align}\label{vague}
	\int_a^b f_n(u)du
	\to \int_a^b f_\infty (u)du.
\end{align}
Let $\color{black}a\in(0,\infty)$ be a continuity point of $f_\infty$. If $\liminf_{n\to\infty} f_n({\color{black}a}) < f_\infty({\color{black}a})$, then there exists $\color{black}b\in(a,\infty)$ such that $\liminf_{n\to\infty}\int_a^b f_n(u)du < \int_a^b f_\infty(u)du$ since for each $n$, the $f_n$ is non-increasing, and $f_\infty$ is continuous at $\color{black}a$.
This contradicts ($\ref{vague}$). Hence
$\liminf_{n\to\infty} f_n({\color{black}a}) \geq f_\infty({\color{black}a})$. By similar arguments we can also obtain $\limsup_{n\to\infty} f_n({\color{black}a}) \leq f_\infty({\color{black}a})$.
}

\begin{proof}[Proof of Lemma \ref{lem:main-a}]
By the assumption and Lemma \ref{lem:lap}, we have for any $u>0$,  
$\lambda\in[0,\infty)^d$ and for any probability measure $\nu\ll\mu$, 
  \begin{align*}
  \color{black}
    \lim_{n\to\infty}\int_X \exp(-u\lambda\cdot S_n/n)d\nu
    =
    \lim_{t\to\infty}\int_X \widetilde{S}_{\lambda,\lfloor ut \rfloor,t}d\nu
    = \bE\left[ \exp(-u\lambda\cdot \zeta)\right].
  \end{align*}
{\color{black}
Therefore, $S_n/n \overset{\cL(\mu)}{\Longrightarrow} \zeta$.}
\end{proof}

\section{Integrating transforms}\label{sec:integrating}

Suppose that $T$ is a CEMPT on $(X,\cA,\mu)$ and that Assumption \ref{ass:ray} holds. Set
\begin{align}
   {\color{black}\varphi}(x)&:=\min\{k\geq1 \;; \;T^k x\in Y\}, \\
   	B_{i,k}&:=Y \cap T^{-1}A_i \cap \{{\color{black}\varphi}=k\}, \notag \\
   	       &=\{x\in Y\;;\; Tx,\cdots,T^{k-1}x\in A_i,\;T^k x\in Y\},
   	       \label{Bik}
\end{align}
for $i=1,\dots,d$ and $k\geq2$. 
The set $B_{i,k}$ consists of elements which start from $Y$, stay on $A_i$ at time $1,\dots,k-1$ and arrive in $Y$ at time $k$. Recall that $Y_n$,  $w_i$ and $\color{black}Q_i$ have been defined by (\ref{Yk}), (\ref{wA}) and (\ref{hatwA}), respectively. 
Note that $Y_n=Y^c\cap \{{\color{black}\varphi} =n\}$ for $n\geq1$
and 
\begin{align*}
	\mu(B_{i,k}) =\int_Y \widehat{T}^k 1_{B_{i,k}}d\mu,
\end{align*}
for $i=1,\dots,d$ and $k\geq2$.
{\color{black}The following lemma was obtained in (6.5) of \cite{TZ}.}

\begin{Lem}\label{lem:excur}
Suppose that $T$ is a CEMPT on $(X,\cA,\mu)$ and that Assumption \ref{ass:ray}  holds. 
Then, for $n\geq1$ and $i = 1,\dots,d$,
	\begin{align}\label{eq:excur}
		\widehat{T}^n 1_{Y_n\cap A_i}=
		  \sum_{k>n} \widehat{T}^k1_{B_{i,k}},
		  \quad \text{$\mu$-a.e.}
	\end{align}
In particular, for $n\geq1$, $s>0$ and $i=1,\dots,d$,
   \begin{align}\label{eq:excur-int}
      w_i(n+1)-w_i(n)&=\mu(Y_n\cap A_i) = \sum_{k>n}\mu(B_{i,k}),
   \\ \label{eq:excur-lap}
     {\color{black}Q_i}(s) &=
	  \sum_{n\geq1}e^{-ns}\sum_{k>n}\mu(B_{i,k}).
   \end{align}
\end{Lem}

{\color{black}

The following lemma is a slight modification of Proposition 3.1 of \cite{Z07}. The proof is almost the same and so we omit it.

\begin{Lem}\label{lem:int}
Let $T$ be a CEMPT on $(X,\cA,\mu)$, let $(R_{n,t}\;;\;n\geq0, \;t>0)$ be a family of measurable functions $R_{n,t}:X\to[0,\infty)$, and let $\mathfrak{H}\subset \{u\in L^1(\mu)\;;\;\text{$\int ud\mu=1$ and $u\geq0$} \}$. Assume that
\begin{itemize}
\item[\rm (i)] for any probability measure $\nu\ll\mu$ and $\varepsilon>0$, 
\begin{align*}
	\nu\Big\{\sup_{t>0}|R_{n,t}\circ T - R_{n,t}|>\varepsilon\Big\}\to 0,
	\quad\text{as $n\to\infty$;}
\end{align*}
\item[\rm (ii)] $\sup\{\|R_{n,t}\|_{L^\infty(\mu)}\;;\;n\geq0,\;t>0\}<\infty$;
\item[\rm (iii)] $\mathfrak{H}$ is strongly precompact in $L^1(\mu)$.
\end{itemize}
Then, it holds that, for any $b>0$, 
\begin{align*}
	\sup_{u,u^*\in\; \bar{\rm co}(\mathfrak{H})}
	\bigg|
	\int_X \sum_{n\geq0}e^{-nbt^{-1}}R_{n,t}(u-u^*)d\mu
	\bigg|
	=
	o(t),
	\quad \text{as $t\to\infty$.}
\end{align*}
where $\bar{\rm co}(\mathfrak{H})$ denotes the closure of the convex hull of $\mathfrak{H}$ in $L^{1}(\mu)$.	
\end{Lem}}

From the above two lemmas we can prove the following. 
{\color{black}
\begin{Cor}\label{cor:int}
Suppose that $T$ is a CEMPT on $(X,\cA,\mu)$ and that Assumptions \ref{ass:ray} and {\color{black}\ref{ass:comp}} hold. Then, for any $\lambda\in{\color{black}[0,\infty)^d}$,  $i=1,\dots,d$ and $a,b>0$, and for any probability measure $\nu\ll\mu$, as $t\to\infty$,
\begin{align}\label{eq:i}
    \int_Y \bigg(\sum_{n\geq2}(1-e^{-nat^{-1}})
        \widehat{T}^{n}1_{B_{i,n}}\bigg)
        &\bigg(\sum_{n\geq0}e^{-nbt^{-1}}
          \widetilde{S}_{\lambda,n,t}\bigg)d\mu
    \notag\\
    \sim\;
     &at^{-1}{\color{black}Q_i}(at^{-1})
     \sum_{n\geq0}e^{-nbt^{-1}}
      \int_X \widetilde{S}_{\lambda,n,t}d\nu,
\end{align}
{\color{black}where $\widetilde{S}_{\lambda,n,t}$ and $Q_i$ have been defined by {\rm (\ref{tildeS})} and {\rm (\ref{hatwA})}, respectively. }
\end{Cor}}
{\color{black}
\begin{proof}
We will write $\widetilde{S}_{n,t}$ instead of $\widetilde{S}_{\lambda,n,t}$ for simplicity.
First of all, we will show that, for any probability measure $\nu\ll\mu$ and $\varepsilon>0$,
\begin{align}\label{cor:int-0}
 	\nu\Big\{\sup_{t>0}|\widetilde{S}_{n,t}\circ T - \widetilde{S}_{n,t}|>\varepsilon\Big\}\to 0,
	\quad\text{as $n\to\infty$}.
\end{align}
We may assume that $\lambda\neq (0,\dots,0)$. Note that
\begin{align}\label{cor:int-1}
	\big|\widetilde{S}_{n,t}\circ T-
	 \widetilde{S}_{n,t}\big|
	\leq\Big( 1-\exp(-t^{-1}\max_{i=1,\cdots,d} \lambda_i )\Big) \widetilde{S}_{n-1,t}\circ T.
\end{align}
We can take $t_0\in(0,\infty)$ so that
\begin{align}\label{cor:int-2}
	\sup_{t\geq t_0} \Big(1-\exp(-t^{-1}\max_{i=1,\cdots,d} \lambda_i)\Big)
	=
	1-\exp(-t_0^{-1}\max_{i=1,\cdots,d} \lambda_i) <\varepsilon.
\end{align}
Since $T$ is a CEMPT, we have 
\begin{align}\label{cor:int-3}
	\sup_{0<t\leq t_0}(\tilde{S}_{n-1,t}\circ T)
	=
	\tilde{S}_{n-1,t_0}\circ T
	\to
	0,
	\quad
	\text{$\mu$-a.e., as $n\to\infty$.}
\end{align}
By (\ref{cor:int-1}), (\ref{cor:int-2}) and (\ref{cor:int-3}), we obtain (\ref{cor:int-0}).

Set
\begin{align*}
  u_t
   :=\frac{\sum_{n\geq1}e^{-nat^{-1}}
        \widehat{T}^{n}1_{Y_n\cap A_i}}
        {Q_i(at^{-1})},
        \quad t>0.	
\end{align*}
By (\ref{hatwA}), we have $u_t\in \bar{\rm co}(\mathfrak{H}_i)$ (see Remark 3.5 of \cite{Z07}).
By Lemma \ref{lem:excur}, we have
\begin{align}\label{eq:ii}
   \sum_{n\geq2}
    (1-e^{-nat^{-1}})\widehat{T}^n 1_{B_{i,n}}
   =e^{at^{-1}}(1-e^{-at^{-1}})
    {\color{black}Q_i(at^{-1})u_t},
    \quad\text{$\mu$-a.e.}
\end{align}
Applying Lemma \ref{lem:int} to $R_{n,t} = \widetilde{S}_{n,t}$ and $\mathfrak{H}=\mathfrak{H}_i\cup\{d\nu/d\mu\}$, we have
\begin{align}
	\bigg|\int_{X} 
        \bigg(\sum_{n\geq0}e^{-nbt^{-1}}
          \widetilde{S}_{n,t}\bigg)\Big(u_t-\frac{d\nu}{d\mu}\Big)d\mu\bigg|
        = o(t),
        \quad\text{as $t\to\infty$.}
\end{align}
It follows immediately that
\begin{align}\label{eq:iii}
	\int_{Y} 
        \bigg(\sum_{n\geq0}e^{-nbt^{-1}}
          \widetilde{S}_{n,t}\bigg)u_td\mu
          \sim 
        \sum_{n\geq0}e^{-nbt^{-1}}
          \int_{X}\widetilde{S}_{n,t}d\nu,
        \quad\text{as $t\to\infty$.}
\end{align}
Here we used the fact that $u_t$ is supported on $Y$ and that $0<\widetilde{S}_{n,t}\leq 1$.
By (\ref{eq:ii}) and (\ref{eq:iii}), we obtain (\ref{eq:i}).
\end{proof}}

%%%%%%%%%%%%%%%%%%%%%%%%%%%%%%%%%%%%%%%%%%%%%%%%%%%%%%%%%%%%%%%%%
%%%%%%%%%%%%%%%%%%%%%%%%%%%%%%%%%%%%%%%%%%%%%%%%%%%%%%%%%%%%%%%%%
%%%%%%%%%%%%%%%%%%%%%%%%%%%%%%%%%%%%%%%%%%%%%%%%%%%%%%%%%%%%%%%%%
\section{Proof of the direct limit theorem}\label{sec:pf}

\begin{Prop}\label{prop:dob-lap}
Let $T$ be a CEMPT on $(X,\cA,\mu)$ and suppose that Assumptions \ref{ass:ray} {\color{black}and \ref{ass:comp}} hold. Then, for any $q>0$, 
$\lambda=(\lambda_1,\dots,\lambda_d)\in {\color{black}[0,\infty)^d}$ and for any probability measure $\nu\ll\mu$,  as $t\to\infty$, 
\begin{align}\label{eq:dob-Lap}
	\int_0^\infty du\:e^{-qu}
      \int_X d\nu \:\widetilde{S}_{\lambda,\lfloor ut \rfloor,t}                  
     \sim
     \frac{\sum_{i=1}^d {\color{black}Q_i}
                 \big((q+\lambda_i)t^{-1} \big)}
          {\sum_{i=1}^d (q+\lambda_i) {\color{black}Q_i}
                 \big((q+\lambda_i)t^{-1} \big)},
\end{align}
{\color{black} where $\widetilde{S}_{\lambda,n,t}$ \color{black}and $Q_i(s)$ have been defined by {\rm (\ref{tildeS})} and {\rm (\ref{hatwA})}, respectively.}
\end{Prop}

\begin{proof}
We will write $\widetilde{S}_{n,t}$ instead of $\widetilde{S}_{\lambda,n,t}$ for simplicity.  
{\color{black}
We have
  \begin{align}\label{eq:a}
     &\bigg|\int_0^\infty du\:e^{-qu}
     \int_X d\nu \:\widetilde{S}_{\lfloor ut \rfloor,t}
     -
     t^{-1} \sum_{n\geq0} e^{-qnt^{-1}}
      \int_X \widetilde{S}_{n,t}d\nu \bigg|
   \notag\\
   &= \bigg|\sum_{n\geq0}\int_{nt^{-1}}^{(n+1)t^{-1}} 
       du \:(e^{-qu}-e^{-qnt^{-1}})
       \int_X d\nu\:\widetilde{S}_{n,t}\bigg| \notag\\
   &\leq  
       \sum_{n\geq0} t^{-1}(e^{-qnt^{-1}}-e^{-q(n+1)t^{-1}}) =t^{-1}.
   \end{align}
Combining (\ref{eq:a}) with the fact that}
\begin{align}\label{ineq:bdd}
\big(q+\max_{i=1,\cdots,d}\lambda_i\big)^{-1}
\leq
   \int_0^\infty du\: e^{-qu}\int_X d\nu\: 
   \widetilde{S}_{\lfloor ut \rfloor,t}
\leq 
q^{-1},	 
\end{align}
{\color{black}we have, 
as $t\to\infty$,}
\begin{align}\label{sim:dob-Lap}
\int_0^\infty du\:e^{-qu}
     \int_X d{\color{black}\nu} \:\widetilde{S}_{\lfloor ut \rfloor,t}
     \sim t^{-1} \sum_{n\geq0} e^{-qnt^{-1}}
      \int_X \widetilde{S}_{n,t}d{\color{black}\nu}.
\end{align}

Recall that the $B_{i,k}$ have been defined by (\ref{Bik}). For each $n\geq1$ and $i=1,\dots,d$, the $\widetilde{S}_{n,t}$ may be represented on $\sum_{k\geq2}B_{i,k}\;(=Y\cap T^{-1}A_i,$ $\mu$-a.e.) as
\begin{equation}\label{eq:dissec}
	\widetilde{S}_{n,t} =
	\begin{cases}
	  e^{-(k-1)\lambda_i t^{-1}}\widetilde{S}_{n-k,t}\circ T^k 
	  &	\text{on }B_{i,k}\;(k=2,\cdots,n),
	  \\
	  e^{-n\lambda_i t^{-1}}
	  & \text{on }\sum_{k>n}B_{i,k}.
	\end{cases}
\end{equation}
By (\ref{eq:dissec}), for each $n\geq1$
and $i=1,\dots, d$, we have
\begin{align*}
	e^{-nqt^{-1}}
	&\int_{Y\cap T^{-1}A_i} \widetilde{S}_{n,t}d\mu 
	\\ 
	&=e^{-nqt^{-1}}\bigg\{
	  \sum_{k=2}^n\int_{B_{i,k}}\Big(
	   e^{-(k-1)\lambda_i t^{-1}}\widetilde{S}_{n-k,t}\circ T^k 
	   \Big)d\mu
	  +
	  e^{-n\lambda_i t^{-1}}
	   \sum_{k>n} \mu(B_{i,k}) \bigg\}
	   \\
	&= e^{\lambda_i t^{-1}}\int_Y \sum_{k=2}^n \Big(e^{-k(q+\lambda_i)t^{-1}}\widehat{T}^k 1_{B_{i,k}} \Big)
	    \Big(e^{-(n-k)q t^{-1}}\widetilde{S}_{n-k,t} \Big)d\mu  
	  \\
	  &\hspace{5mm}
	  +e^{-n(q+\lambda_i)t^{-1}}
	   \sum_{k>n} \mu(B_{i,k}).
\end{align*}
On one hand, summing over $n\geq1$ and $i=1,\dots,d$, we have by (\ref{eq:excur}) and (\ref{eq:excur-lap}),
\begin{align}\label{eq:1}
	\sum_{n\geq1}
	&e^{-nqt^{-1}}
	 \int_{Y\cap T^{-1}Y^c}  \widetilde{S}_{n,t}d\mu \notag\\
       &= \sum_{i=1}^d\;e^{\lambda_i t^{-1}}
       \int_Y
	   \bigg( \sum_{n\geq2}e^{-n(q+\lambda_i)t^{-1}}
	   \widehat{T}^n 1_{B_{i,n}} \bigg)
	   \bigg( \sum_{n\geq0} e^{-nqt^{-1}}\widetilde{S}_{n,t}
	   \bigg)d\mu 
	   +\sum_{i=1}^d
	   {\color{black}Q_i}\big((q+\lambda_i)t^{-1}\big)
	\notag\\
	&= \sum_{n\geq0} e^{-nqt^{-1}}
	      \int_{Y^c\cap T^{-1}Y}
	       \widetilde{S}_{n,t}\circ Td\mu
	   \notag\\
	   &\quad\quad
	   -
	   \sum_{i=1}^d
       \int_Y
	   \bigg( \sum_{n\geq2}(1-e^{-n(q+\lambda_i)t^{-1}})
	   \widehat{T}^n 1_{B_{i,n}} \bigg)
	   \bigg( \sum_{n\geq0} e^{-nqt^{-1}}\widetilde{S}_{n,t}
	   \bigg)d\mu 
	   \notag\\
	   &\quad\quad
	   +
	   \sum_{i=1}^d
	   {\color{black}Q_i}\big((q+\lambda_i)t^{-1}\big)+O(1),
	   \quad\text{as $t\to\infty$}.
\end{align}
Here we used the following estimates: 
\begin{align*}
\bigg\|\sum_{n\geq2}(e^{-n(q+\lambda_i)t^{-1}}-1)
	   \widehat{T}^n 1_{B_{i,n}}\bigg\|_{L^\infty (\mu)}
%	&=\bigg|\bigg|(1-e^{-(q+\lambda_i)t^{-1}})\sum_{k\geq1}
%	e^{-k(q+\lambda_i)t^{-1}}
%	   \widehat{T}^k 1_{Y_k\cap A_i}
%     \\
%      &\quad+(1-e^{-(q+\lambda_i)})\widehat{T} 1_{Y_1 \cap A_i}
%       \bigg|\bigg|_{L^\infty (\mu)}
%	  \\
	&=O(1),
	\quad\text{as $t\to\infty$},
	   \\
     \bigg\|\sum_{n\geq0} e^{-nqt^{-1}}\widetilde{S}_{n,t}\bigg\|_{L^\infty(\mu)}
    &=O(t),
    \quad\text{as $t\to\infty$}.	
\end{align*}
These estimates are easily deduced from (\ref{eq:ii}) and $0\leq\widehat{T}^{n}1_{Y_n\cap A_i},\widetilde{S}_{n,t}\leq1$.
On the other hand, since $T$ is $\mu$-preserving and
\begin{align*}
	\big|\widetilde{S}_{n,t}-
	 \widetilde{S}_{n,t}\circ T\big|
	\leq 1-\exp(-t^{-1}\max_{i=1,\cdots,d} \lambda_i ),
\end{align*}
we have the following evaluation:
\begin{align}
	\int_{Y\cap T^{-1}Y^c}  
	   &\widetilde{S}_{n,t}d\mu
	   -\int_{Y^c\cap T^{-1}Y}
	       \widetilde{S}_{n,t}\circ Td\mu
    \notag\\
     &=\int_{Y}  
	   \widetilde{S}_{n,t}d\mu
	   -\int_{T^{-1}Y}
	       \widetilde{S}_{n,t}\circ Td\mu
	   -\int_{Y\cap T^{-1}Y}
	   \big( \widetilde{S}_{n,t}
	    -\widetilde{S}_{n,t}\circ T\big)d\mu
	 \notag\\
	 &=-\int_{Y\cap T^{-1}Y}
	   \big( \widetilde{S}_{n,t}
	    -\widetilde{S}_{n,t}\circ T\big)d\mu
	 \notag\\ \label{eq:2}
	 &=O(t^{-1}), \quad\text{as $t\to\infty$, uniformly in $n$}.
\end{align}
We now combine (\ref{eq:1}) and (\ref{eq:2}) and thus we obtain
\begin{align}\label{eq:3}
	\sum_{i=1}^d
       \int_Y
	   \bigg( \sum_{n\geq2}(1-e^{-n(q+\lambda_i)t^{-1}})
	   &\widehat{T}^n 1_{B_{i,n}} 
	   \bigg)
	   \bigg( \sum_{n\geq0} e^{-nqt^{-1}}\widetilde{S}_{n,t}
	   \bigg)d\mu
	   \notag
	   \\ 
	&
	=\sum_{i=1}^d {\color{black}Q_i}\big((q+\lambda_i)t^{-1}\big)
	   +O(1),\quad\text{as $t\to\infty$.}
\end{align}

By Corollary \ref{cor:int} and (\ref{sim:dob-Lap}), we obtain for each $i=1,\dots,d$, as $t\to\infty$,
\begin{align}\label{eq:4}
   \int_Y
	   \bigg( \sum_{n\geq2}(1-e^{-n(q+\lambda_i)t^{-1}})
	   &\widehat{T}^n 1_{B_{i,n}} 
	   \bigg)
	   \bigg( \sum_{n\geq0} e^{-nqt^{-1}}\widetilde{S}_{n,t}
	   \bigg)d\mu
     \notag\\
   &\sim
      {\color{black}
      (q+\lambda_i)Q_i
      \big( (q+\lambda_i)t^{-1} \big)
      \int_0^\infty du\:e^{-qu}
      \int_X d\nu \:\widetilde{S}_{\lfloor ut \rfloor,t}}.   
\end{align}
Note that ${\color{black}Q_i}(s)\to\infty$, as $s\downarrow0$. Summing up {\color{black}(\ref{eq:4})} over $i=1,\dots,d$, and combining it with {\color{black}(\ref{eq:3})} and (\ref{ineq:bdd}), we obtain (\ref{eq:dob-Lap}).
\end{proof}

We now prove Theorem \ref{thm:main}.

\begin{proof}[Proof of Theorem \ref{thm:main}]
By Assumption \ref{ass:reg} and Remark \ref{rem:reg-lap}, for any $q>0$ and $\lambda_i\geq0$,  
\begin{align*}
   \lim_{t\to\infty}
    \frac{{\color{black}Q_i}\big((q+\lambda_i)t^{-1}\big)}
          {{\color{black}Q}(t^{-1})} 
   = \beta_i(q+\lambda_i)^{-(1-\alpha)}.
\end{align*}
Hence Proposition \ref{prop:dob-lap} implies for any $q>0$, $\lambda=(\lambda_1,\dots,\lambda_d)\in{\color{black}[0,\infty)^d}$ and for any probability measure $\nu\ll\mu$,
\begin{align}
  \lim_{t\to\infty}
  \int_0^\infty du\:e^{-qu}
      \int_X d\nu \: \widetilde{S}_{\lambda,\lfloor ut \rfloor,t}
  =
  \frac{\sum_{i=1}^d \beta_i(q+\lambda_i)^{-(1-\alpha)}}
          {\sum_{i=1}^d \beta_i (q+\lambda_i)^\alpha }.
\end{align}
By Proposition \ref{prop:dob-lap-arc} and Lemma \ref{lem:main-a}, we obtain $S_n/n \overset{\cL(\mu)}{\Longrightarrow} \zeta_{\alpha,\beta}$.
\end{proof}

%%%%%%%%%%%%%%%%%%%%%%%%%%%%%%%%%%%%%%%%%%%%%%%%%%%%%%%%%%%%%%%%%
%%%%%%%%%%%%%%%%%%%%%%%%%%%%%%%%%%%%%%%%%%%%%%%%%%%%%%%%%%%%%%%%%
%%%%%%%%%%%%%%%%%%%%%%%%%%%%%%%%%%%%%%%%%%%%%%%%%%%%%%%%%%%%%%%%%
\section{Proof of the inverse limit theorem}\label{sec:pf2}

For the proof of the inverse limit theorem, we mimic the method of Theorem 1 of Watanabe \cite{W95}.
For $0<a<b<\infty$, let the map $f:[0,\infty]\to [b^{-1},a^{-1}]$ be defined by
\begin{align*}
f(x):=\frac{x+1}{bx+a}=\frac{1}{b}+\frac{(b-a)}{b(bx+a)}.	
\end{align*}
Then $f$ is a homeomorphism. Hence for $(x_n)_{n\geq0}\subset {\color{black}[0,\infty)}$ and $x\in[0,\infty]$, $x_n\to x$ if and only if $f(x_n)\to f(x)$.
Recall that $e^{(i)}=(1_{\{i=j\}})_{j=1}^d \in [0,1]^d$ for $i=1,\dots,d$.

\begin{proof}[Proof of Theorem \ref{thm:main2}]
We investigate the asymptotics of the occupation times on the first ray $A_1$, say. For this purpose, we take a special look at the first ray. Set for $t,q,r>0$,
\begin{align}
	W_t(q,r):=
	\frac{{\color{black}Q_1}\big((q+r)t^{-1}\big)}
    {\sum_{i=2}^d{\color{black}Q_i}\big(qt^{-1}\big)}.
\end{align}
Assume that $S_n/n \overset{\nu_0}{\Longrightarrow} \zeta$,
then for any $q>0$ and $\lambda=(\lambda_1,\dots,\lambda_d)\in{\color{black}[0,\infty)^d}$, the left-hand side of (\ref{eq:dob-Lap}) converges as $t\to\infty$, 
and this implies for any $q,r>0$, 
\begin{align*}
	W(q,r):=\lim_{t\to\infty} W_t(q,r)\in[0,\infty]
\end{align*}
exists, since we have, by substituting $re^{(1)}$ for $\lambda$ in the right-hand side of (\ref{eq:dob-Lap}), 
\begin{align*}
	\frac{{\color{black}{Q}_1}\big((q+r)t^{-1}\big) 
	      + \sum_{i=2}^d {\color{black}Q_i} (qt^{-1})}
         {(q+r){\color{black}Q_1}\big((q+r)t^{-1}\big)
          + q \sum_{i=2}^d {\color{black}Q_i} (qt^{-1})}
    =
    \frac{W_t(q,r)+1}{(q+r)W_t(q,r)+q}.
\end{align*}
Note that $W(q,\cdot)$ is a non-increasing function for any $q>0$ and $W(pq,r)=W(q,p^{-1}r)$ for any $p,q,r>0$. 
We divide the study into the following three distinct cases:
\begin{enumerate}
 \item[(i)] the case where $W(q_0,r_0)=\infty$ for some $q_0,r_0>0$.

In this case, by the above remark, $W(q,r)=W(q_0,q_0q^{-1}r)=\infty$ for any $q\geq q_0$ and $r \leq r_0$.
Hence by Proposition \ref{prop:dob-lap}, for any $q\geq q_0$ and $r\leq r_0$ and for any probability measure $\nu\ll\mu$,
%In this case, we can easily deduce that we have ${\color{black}Q_1}(s)/{\color{black}Q}(s)\to 1$, as $s\downarrow 0$.
%Furthermore, by the above remark, $W(q,r)=W(q_0,q_0q^{-1}r)=\infty$ for any $q\in [q_0,\infty)$ and $r \in [0,r_0]$. Hence, by Proposition \ref{prop:dob-lap}, for any $q\in [q_0,\infty)$ and $r \in [0, r_0]$ and for any probability measure $\nu\ll\mu$,
\begin{align}
	\int_0^\infty du\:e^{-qu}
      \int_X d\nu \:\exp\pmat{-re^{(1)}\cdot S_{\lfloor ut \rfloor}/t }
      \to
      (q+r)^{-1}, \quad\text{as $t\to\infty$}.     
\end{align}
{\color{black}By a slight modification of Lemma \ref{lem:main-a}, we have} $n^{-1}\sum_{k=1}^n 1_{A_1}\circ T^k \overset{\cL(\mu)}{\Longrightarrow}1$.
{\color{black}Hence, by Proposition \ref{prop:dob-lap}, we have} ${\color{black}Q_1}(s)/{\color{black}Q}(s)\to 1$, as $s\downarrow0$. 
By Remark \ref{rem:birkoff},
\begin{align}\label{eq:conv-to-1}
\sum_{i=1}^d n^{-1} \sum_{k=1}^n 1_{A_i} \circ T^k \to 1, 
     \quad\text{$\mu$-a.e., as $n \to \infty$.} 
\end{align}
Hence we obtain $S_n/n \overset{\cL(\mu)}{\Longrightarrow}e^{(1)}$.
   
 \item[(ii)] the case where $W(q_0,r_0)=0$ for some $q_0,r_0>0$. 
 
 By the similar arguments as (i), we may obtain $n^{-1}\sum_{k=1}^n1_{A_1}\circ T^k \overset{\cL(\mu)}{\Longrightarrow}0$ and ${\color{black}Q_1}(s)/{\color{black}Q}(s)\to0$, as $s\downarrow 0$.

 \item[(iii)] the case where $W(q,r)\in(0,\infty)$ for any $q,r>0$.

For $p>0$, set $q:=2^{-1}\min\{p,1\},\; r:=p-q,$ and $r':=1-q$. Then
\begin{align*}
	\frac{{\color{black}Q_1}(ps)}{{\color{black}Q_1}(s)}
	=\frac{W_{s^{-1}}(q,r)}{W_{s^{-1}}(q,r')}
	\to
	\frac{W(q,r)}{W(q,r')} \in (0,\infty),
	\quad \text{as $s\downarrow0$},
\end{align*}
Hence $\color{black}Q_1$ is regularly varying at 0 (see Theorem 1.4.1 of \cite{BGT}). Furthermore, by (\ref{eq:excur-lap}) and since $n\mapsto\sum_{k>n}\mu(B_{i,k})$ is positive and non-increasing, the index of regular variation must be in $[-1,0]$, i.e., ${\color{black}Q_1}\in\cR_{-(1-\alpha)}(0+)$ for some $\alpha\in[0,1]$. By the definition and assumption of $W(q,r)$, this implies, for some constant ${\color{black}c\in(0,\infty)}$,
\begin{align*}
 \sum_{i=2}^d{\color{black}Q_i}(s)
   \sim c {\color{black}Q_1}(s),
   \quad \text{as $s\downarrow0$}. 	
\end{align*}
Therefore ${\color{black}Q}(s)\in\cR_{-(1-\alpha)}(0+)$ and ${\color{black}Q_1}(s)\sim (1+c)^{-1} {\color{black}Q}(s)$, as $s \downarrow 0$.

\end{enumerate}

The above argument works for the other rays. For $j=1,\dots,d$ and for $q,r>0$,
\begin{align}
  W^{(j)}(q,r)
	:=\lim_{t\to\infty}
	  \frac{{\color{black}Q_j} \big((q+r)t^{-1}\big)}
	{\sum_{i\neq j}{\color{black}Q_i} (qt^{-1})}
    \in[0,\infty]
\end{align}
exists. 

If $W^{(j)}(q_0,r_0)=\infty$ for some $j=1,\dots,d$ and for some $q_0,r_0>0$, then by the same arguments as (i), $S_n/n \overset{\cL(\mu)}{\Longrightarrow}e^{(j)}$.

Let us consider the other cases. If for any $j=1,\dots,d$, the function $W^{(j)}$ has zeros, then by the same argument as (ii) we obtain 
$S_n/n\overset{\cL(\mu)}{\Longrightarrow}0$, but this contradicts (\ref{eq:conv-to-1}).
So there exists $j\in\{1,\dots,d\}$ such that $W^{(j)}(q,r)\in(0,\infty)$ for any $q,r>0$.
Then by the same arguments as (ii) and (iii), and by Remark \ref{rem:reg-lap}, 
there exist some constants $\alpha \in [0,1]$ 
and $\beta = (\beta_1, \dots, \beta_d) \in[0,1)^d$ with $\sum_{i=1}^d\beta_i=1$
such that (\ref{eq:reg}) and (\ref{eq:bal}) hold.
Hence by Theorem \ref{thm:main}, we obtain $S_n/n \overset{\cL(\mu)}{\Longrightarrow} \zeta_{\alpha,\beta}$.
\end{proof}

\end{document}